\documentclass[a4paper,12pt,reqno]{article}

\usepackage{hyperref}

\usepackage{pgf,tikz}
\usetikzlibrary[patterns]
\usepackage{mathrsfs}
\usetikzlibrary{arrows}
\usepackage{xcolor}

\usepackage{fancyhdr}
\usepackage{stmaryrd}
\usepackage[scale=0.75]{geometry}
\usepackage{mathtools}
\usepackage{color}
\usepackage{url}
\usepackage{amsmath}
\usepackage{amsthm}
\usepackage{amssymb}
\usepackage{amscd}          			

\usepackage{bbm}            			
\usepackage{mathrsfs}           		

\usepackage[T1]{fontenc}       		
\usepackage[english]{babel}     		
\usepackage[utf8]{inputenc}

\usepackage{graphicx}          		

\usepackage{verbatim}           		

\theoremstyle{plain}
\newtheorem{Theorem}{Theorem}[section]

\newtheorem{Proposition}[Theorem]{Proposition}
\newtheorem{Conjecture}[Theorem]{Conjecture}

\theoremstyle{definition}
\newtheorem{Definition}[Theorem]{Definition}

\theoremstyle{remark}
\newtheorem{Remark}[Theorem]{Remark}
\newtheorem{Claim}{Claim}

\newcommand{\N}{\mathbb{N}}     
\newcommand{\R}{\mathbb{R}}     

\newcommand{\Z}{\mathbb{Z}}         

\DeclareMathOperator{\diam}{diam}			

\usepackage[mathcal]{eucal}					

\renewcommand{\geq}{\geqslant}
\renewcommand{\leq}{\leqslant}
\renewcommand{\epsilon}{\varepsilon}

\newcommand{\liminfe}{\mathop{\underline{\lim}}}

\title{\textbf{(Un)boundedness of directional maximal operators through a notion of ``Perron capacity'' and an application}}

\author{Emma D'Aniello, Anthony Gauvan and Laurent Moonens}

\begin{document}

\maketitle

\begin{abstract}
We introduce the notion of \textit{Perron capacity} of a set of slopes $\Omega \subset \mathbb{R}$. Precisely, we prove that if the Perron capacity of $\Omega$ is finite then the directional maximal operator associated $M_\Omega$ is not bounded on $L^p(\mathbb{R}^2)$ for any $1 < p < \infty$. This allows us to prove that the set $$\Omega_{ \boldsymbol{e}} =\left\{ \frac{\cos n}{n}: n\in \mathbb{N}^* \right\}$$ is not finitely lacunary which answers a question raised by A. Stokolos.
\end{abstract}


\section{Introduction}
It is well known, since the original work of Jensen, Marcinkiewicz and Zygmund \cite{JMZ}, that Lebesgue's differentiation formula, \emph{i.e.} the equality for a.e.\ $x\in\R^n$:
\begin{equation}\label{eq.leb}
f(x)=\lim_{\begin{subarray}{c} R\ni x\\\diam R\to 0\end{subarray}} \frac{1}{|R|} \int_R f,
\end{equation}
holds along $n$-dimensional \emph{rectangles} $R$ parallel to the coordinate hyperplanes (called standard rectangles), whenever one has $f\in L\log^{n-1}L (\R^n)$~---~the latter Orlicz space being, moreover, the largest for which the above formula holds if one does not assume any extra hypothesis on the rectangles $R$; the beautiful survey by A.~Stokolos \cite{STOKOLOSAF} provides a quite exhaustive state of the art on the topic~---~see also \cite{DM2017} for sufficient conditions (of geometric nature) on $\mathcal{R}$ ensuring that $L\log^{n-1}L(\R^n)$ is the largest space for which \eqref{eq.leb} holds whenever $R$ is restricted to belong to the class $\mathcal{R}$.

In dimension $n=2$, the nature of the problem changes when one allows rectangles $R$ to rotate around one of their vertices, requiring that the slope of their longest side belongs to a fixed set $\Omega$ (call $\mathcal{R}_\Omega$ the family of such rectangles in the plane). It has been shown \emph{e.g.} by Cordoba and Fefferman \cite{CF} that when $\Omega=\{\omega_k:k\in\N\}$ is a lacunary sequence (see below for a definition), then formula \eqref{eq.leb} holds for any $f\in L^p(\R^2)$ if one has 
$p \in [2,  +\infty[$. This was later generalized for arbitrary $1<p<\infty$ by Nagel, Stein and Wainger \cite{NSW}. Given a set of slopes $\Omega$, the question of finding Orlicz spaces (the largest possible) beyond $L^p$ spaces, for all elements of which formula \eqref{eq.leb} holds is often delicate; issues in these directions have been studied in works involving some of the current authors, see \cite{DMR} and \cite{DM} for example.

Concerning the possibility of \eqref{eq.leb} holding for all $f\in L^p(\R^2)$, $1<p\leq\infty$, it has been shown \emph{e.g.} by Katz \cite{KATZ}, Bateman and Katz \cite{BK}, Hare \cite{HARE} and more recently by Bateman \cite{BATEMAN} that Cantor sets (and even uncountable sets) never give rise to formula \eqref{eq.leb} for all $f\in L^p(\R^2)$, $1<p\leq \infty$~---~Bateman's work \cite{BATEMAN} even gives a complete characterization, even though not always simple to implement \emph{in practice} (see below), of sets $\Omega$ for which it does hold in any $L^p(\R^2)$, $1<p\leq\infty$: those are sets one will call \emph{finitely lacunary} in the sequel.

Before to state more precisely Bateman's result, let us mention an important tool in its proof: the possibility, when a set of slopes $\Omega$ is too large, to find finite families of rectangles in $\mathcal{R}_\Omega$ such that the union of their (centered) homothetic expansions occupies a much larger area in the plane than their union does; we call this a \emph{Kakeya blow}.

When no restriction is made on directions of rectangles in play, the possibility of exhibiting a Kakeya blow is a standard result in measure theory (see Busemann and Feller's work \cite{BF} for a first construction of this type):
\begin{Theorem}[Kakeya Blow]
For any large $A \gg 1$, there exists a finite family of rectangles $\left\{ R_i \right\}_{i \leq N} $ such that we have $$A\left| \bigcup_{i \leq N} R_i \right| \leq \left| \bigcup_{i \leq N} 4R_i \right|.$$ Here, $4R$ stands for the $4$-fold dilation of $R$ by its center.
\end{Theorem}
Kakeya blows have deep implications in harmonic analysis: for example, a concrete construction proving the above theorem allowed Fefferman to disprove the Disc multiplier conjecture in \cite{FDMC}.

A natural question, given a set of slopes $\Omega$, then becomes the following: can we still construct Kakeya blows using \emph{only} rectangles in $\mathcal{R}_\Omega$, namely rectangles the longest side of which makes an angle $\omega \in \Omega$ with the horizontal axis (see also Stokolos \cite{STOKOLOSM} and Hagelstein and Stokolos \cite{HSM} concerning the links between directional maximal operators and multipliers)? In his study of the so-called \textit{directional maximal operator} $M_\Omega$ defined for a locally integrable $f : \mathbb{R}^2 \rightarrow \mathbb{R}$ and $x \in \mathbb{R}^2$ as $$M_\Omega f(x) := \sup_{ x \in R \in \mathcal{R}_\Omega} \frac{1}{\left| R\right|}\int_R \left| f\right|,$$
Bateman \cite{BATEMAN} introduced a notion we shall call \textit{finitely lacunary set} and proved the following Theorem.

\begin{Theorem}[Bateman]\label{T0}
The following conditions are equivalent \begin{itemize}
    \item the set $\Omega$ is not finitely lacunary ;
    \item for any $A \gg 1$, there exists a finite family of rectangles $\{ R_i \}$ such that the longest side of each makes an angle $\omega \in \Omega$ with the $x$-axis and also satisfying $A\left| \bigcup_{i \in I} R_i \right| \leq \left| \bigcup_{i \in I} 4R_i \right|$ ; 
    \item the operator $M_\Omega$ is not bounded on $L^p(\mathbb{R}^2)$ for any $1 < p < \infty$.
\end{itemize}
Moreover, the following conditions are also equivalent \begin{itemize}
    \item the set $\Omega$ is finitely lacunary
    \item the operator $M_\Omega$ is bounded on $L^p(\mathbb{R}^2)$ for any $1 < p < \infty$.
\end{itemize}
\end{Theorem}
In the latter theorem, it needs to be explained what is meant by ``finitely lacunary''. We recall it here, following the presentation by Kroc and Pramanik \cite{KP}. We start by defining the notion of \textit{lacunary sequence} and then the notion of \textit{lacunary set of finite order}.

\begin{Definition}[Lacunary sequence]
We say that a sequence of real numbers $L=(L_k)$ is a lacunary sequence converging to $\ell \in \mathbb{R}$ when there holds $$|\ell-\ell_{k+1}| \leq \frac{1}{2}|\ell - \ell_k|$$ for any $k$.
\end{Definition}

For example the sequences $L_1 := \left(\frac{1}{2^k}\right)_{k\geq 2}$ and $L_2 := \left(\frac{1}{k!}\right)_{k\geq 4}$ are lacunary.

\begin{Definition}[Lacunary set of finite order]
A lacunary set of order $0$ in $\mathbb{R}$ is a set which is either empty or a singleton. Recursively, for $N \in \mathbb{N}^*$, we say that a set $\Omega$ included in $\mathbb{R}$ is a lacunary set of order at most $N+1$~---~and write $\Omega \in \Lambda(N+1)$~---~if there exists a lacunary sequence $L$ with the following properties : for any $a,b \in L$ such that $a < b$ and $(a,b) \cap L = \emptyset$, the set $\Omega \cap (a,b)$ is a lacunary set of order at most $N$ \textit{i.e.} $\Omega \cap (a,b) \in \Lambda(N)$.
\end{Definition}

For example the set $$\Omega := \left\{ \frac{1}{2^k} + \frac{1}{4^l} : k,l \in \mathbb{N}, l \leq k \right\} $$ is a lacunary set of order $2$. In this case, observe that the set $\Omega$ cannot be re-written as $\Omega = \left\{ \omega_k : k \in \mathbb{N}^* \right\}$ where $(\omega_k)$ is a monotone sequence, since it has several points of accumulation. We can finally give a definition of a finitely lacunary set.

\begin{figure}[h!]
\centering
\includegraphics[scale=1]{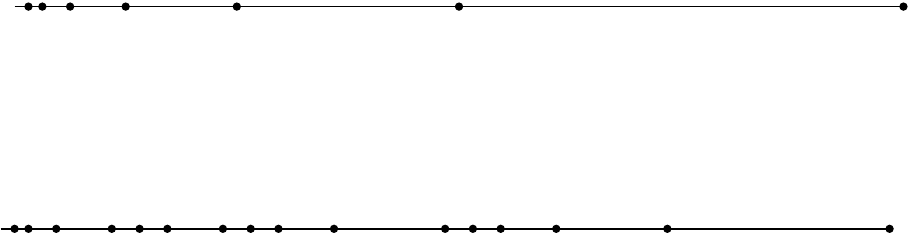}
\caption{A representation of a lacunary sequence and a lacunary set of order $2$.}  
\end{figure}

\begin{Definition}[Finitely lacunary set]
A set $\Omega$ in $\mathbb{R}$ is said to be \textit{finitely lacunary} if there exists a finite number of set $\Omega_1, \dots, \Omega_M$ which are lacunary of finite order such that $$ \Omega \subset \bigcup_{k \leq M } \Omega_k.$$
\end{Definition}

Even if Theorem~\ref{T0} gives a full characterization of the behavior of directional maximal operators, it is usually hard \emph{in practice} to decide whether a set of slopes $\Omega$ is finitely lacunary. For example, it was known that the set of slopes $\Omega_0$ defined by $$\Omega_0 := \left\{ \frac{1}{n} : n \in \mathbb{N}^* \right\}$$ generates a directional maximal operator $M_{\Omega_0}$ which is not bounded on $L^p(\mathbb{R}^2)$ for any $1 < p < \infty$ (see \emph{e.g.} \cite[Theorem~2.1]{DG}); using Theorem~\ref{T0} one can hence assert that the set of directions $\Omega_0$ is not finitely lacunary. In the case where the set of directions $\Omega$ is a discrete set which can be ordered in $\R_+^*$, more precisely in the case where $\Omega = \{\frac{1}{u}: u \in U\}$ and we can write $U = \left\{  u_k:k\in\N^*  \right\} \subset \R_+^*$ for an increasing sequence $(u_k)_{k\in\N^*}$, Hare and R\"{o}nning \cite{HR} provided a \textit{quantitative} sufficient condition on $U$ which ensures that \textit{the maximal operator associated $M_U$ is not bounded on $L^p(\mathbb{R}^2)$.} Precisely, they proved the following Theorem.

\begin{Theorem}[Hare-R\"{o}nning \cite{HR}]\label{T5}
Assume that $U=(u_k)_{k\in\N^*} \subset (0,\pi/2)$ is an increasing sequence of directions and that one has $$ G_U := \sup_{\begin{subarray}{c}k \geq 1\\ l \leq k\end{subarray}} \left( \frac{u_{k+2l} - u_{k+l}}{u_{k+l} - u_{k}} + \frac{u_{k+l} - u_{k}}{u_{k+2l} - u_{k+l}} \right) < \infty.$$ Then for any $p \in (1,\infty)$, one has $\| M_U\|_p = \infty$.
\end{Theorem}

In particular, applying Bateman's Theorem, it follows that if we have $G_U < \infty$ then $U$ is not finitely lacunary~---~other uses of Perron trees have also been made in works involving the current authors, see \emph{e.g.}. \cite{DGMR}, \cite{G1} and \cite{G3}. Hence, usually  one does not prove \textit{directly} that a set of slopes $\Omega$ is not finitely lacunary (or that it is): in general, it is a consequence of the (un)-boundedness of the associated maximal operator $M_\Omega$~---~obtained by other methods~---~and Bateman's Theorem.

Our results are motivated by the following question raised by A. Stokolos : what can be said (for example) about the set of directions $\Omega_{ \boldsymbol{e}}$ defined as $$\Omega_{ \boldsymbol{e}} := \left\{ \frac{\cos n}{n} : n \in \mathbb{N}^* \right\} \ ?$$ Can one prove that it is finitely lacunary or not? In this text, we prove it is \emph{not} finitely lacunary.

\section{Results}

We denote by $\mathcal{R}$ the set containing all rectangles of the plane and for $R \in \mathcal{R}$, we denote by $\omega_R \in [0,\pi)$ the angle that the longest side of the rectangle $R$ makes with the $x$-axis. For any set $\Omega \subset \mathbb{R}$ (which should be thought as a set of \textit{slopes}) we define the family of rectangle $\mathcal{R}_\Omega$ as $$ \mathcal{R}_\Omega := \left\{ R \in \mathcal{R} : \tan(\omega_R) \in \Omega \right\}.$$ We then define the \textit{directional maximal operator} $M_\Omega$ as $$ M_\Omega f(x) := \sup_{x \in R \in \mathcal{R}_\Omega} \frac{1}{\left| R \right|}\int_R |f|.$$ We finally define the \textit{Perron capacity} of a set of slopes $\Omega$ included in $\mathbb{R}$ as follow.

\begin{Definition}[Perron capacity of a set of slopes]
For any set of slopes $\Omega \subset \mathbb{R}$, we define its \textit{Perron capacity} as $$P_\Omega :=\liminfe_{N \to\infty} \inf_{ \begin{subarray}{c}U \subset \Omega\\ \#U = 2^N \end{subarray}}  G_U \in [2,\infty]$$ where $G_U = \sup_{\begin{subarray}{c}k,l \geq 1\\ k+2l \leq 2^N\end{subarray}} \left( \frac{u_{k+2l} - u_{k+l}}{u_{k+l} - u_{k}} + \frac{u_{k+l} - u_{k}}{u_{k+2l} - u_{k+l}} \right)$ if $U = \{ u_1,\dots, u_{2^N} \}$ with $u_1<u_2<\cdots <u_N$.
\end{Definition}

Our first result is the following: it gives a \textit{sufficient} and \textit{quantitative} condition on an arbitrary set $\Omega$ which ensures that the associated maximal operator is unbounded on $L^p(\mathbb{R}^2)$ for any $1 < p < \infty$. In contrast with Hare and Ronning Theorem, we do not assume that our set of slopes is ordered.

\begin{Theorem}\label{T6}
Fix any set of slopes $\Omega \subset \mathbb{R}$ and suppose that we have $$P_\Omega < \infty.$$ Then for any $p \in (1,\infty)$, one has $\| M_{\Omega}\|_p = \infty$.
\end{Theorem}

Loosely speaking, the fact that $P_\Omega < \infty$ indicates whether the set $\Omega$ contains arbitrary large subsets which are (more or less) uniformly distributed in $\mathbb{R}$. Our second result is an application of Theorem \ref{T6} and deals with the set of slopes $\Omega_{\boldsymbol{e}}$ defined as $$\Omega_{ \boldsymbol{e}} := \left\{ \frac{\cos n}{n} : n \in \mathbb{N}^* \right\}.$$ Namely, we prove the following Theorem.

\begin{Theorem}\label{T1}
The Perron capacity of the set $\Omega_{ \boldsymbol{e}}$ is finite \textit{i.e.} we have $P(\Omega_{ \boldsymbol{e}}) < \infty$.
\end{Theorem}

Hence, as a consequence of Theorems \ref{T1} and \ref{T6} we know that~---~for any $1 < p < \infty$~---~we have $$\|M_{\Omega_{ \boldsymbol{e}}}\|_p = \infty.$$ As an application of Bateman's Theorem, we can hence say that the set $\Omega_{ \boldsymbol{e}}$ is not finitely lacunary.

\section{Proof of Theorem \ref{T6}}

Recall that for any set $U := \left\{ u_1,\dots,  u_{2^N} \right\} \subset \mathbb{R}_+$ with $u_1<u_1<\cdots <u_{2^N}$ (so that $U$ has $2^N$ elements), one define its \textit{Perron factor} as $$G_U :=  \sup_{\begin{subarray}{c}k,l \geq 1\\ k+2l \leq 2^N\end{subarray}} \left( \frac{u_{k+2l} - u_{k+l}}{u_{k+l} - u_{k}} + \frac{u_{k+l} - u_{k}}{u_{k+2l} - u_{k+l}} \right) \in (0,\infty).$$ The following proposition is a careful reading of Hare and R\"{o}nning's work \cite{HR}.

\begin{Proposition}
There exists $\epsilon_0\in (0,1)$ such that for any $\alpha\in [1-\epsilon_0,1)$ there exists a set $X\subseteq\R^2$ for which one has: $$ \left| X \right| \leq  \left[ \alpha^{2N} + G_U(1-\alpha)^2 \right]\left| \{ M_U \mathbf{1}_X > \frac{1}{4} \} \right|.$$ 
\end{Proposition}

Hence for any $0<\alpha<1$ close enough to $1$, this gives us the following lower bound for any $p \in (1,\infty)$: $$\|M_U \|_p^p \gtrsim_p \frac{1}{ \alpha^{2N} + G_U(1-\alpha)^2},$$
where $\gtrsim_p$ means that the inequality holds up to a multiplicative constant that does only depend on $p$.

Indeed, we have $$ \| M_U \mathbf{1}_X \|_p^p \geq \frac{1}{4^p} \left| \{ M_U \mathbf{1}_X > \frac{1}{4} \} \right|  \gtrsim_p \frac{\left| X \right|}{ \alpha^{2N} + G_U(1-\alpha)^2 }.$$ This allows us now to easily conclude the proof of Theorem~\ref{T6}. Indeed, fix an arbitrary set $\Omega \subset \mathbb{R}$ and suppose that its Perron capacity is finite \textit{i.e.} assume that one has: $$P_\Omega < \infty.$$ By definition of $P_\Omega$, there thus exists a strictly increasing sequence of integers $\left\{ N_k: k \in \mathbb{N}^* \right\}$ such that for any $k$ there is a set $U_k \subset \Omega$ such that $G(U_k) < 2P_\Omega$ and $\#U_k = 2^{N_k}$. Since there holds $U_k \subset \Omega$, we obtain, for any $0<\alpha<1$ sufficiently close to $1$ and any $k \geq 1$: $$\|M_{ \Omega } \|_p^p \geq \| M_{U_k} \|_p^p \gtrsim \frac{1}{\alpha^{2N_k} + 2P_\Omega(1-\alpha)^2}.$$ Since this holds for any $ k \geq 1$ and any $\alpha$ close to $1$, this implies that we have $\|M_\Omega\|_p = \infty$.

\section{Homogeneous sets}

We fix an arbitrary set $U$ in $\mathbb{R}^*_+$ whose cardinal is $2^N$ for some $N \in \mathbb{N}^*$. The following proposition shows that it is equivalent for $U$ to be uniformly distributed and its Perron factor to equal $2$.

\begin{Proposition}
We have $G_U = 2$ if and only if the elements of $U$ are in arithmetic progression \emph{i.e.} $U$ is of the form $$U = \{ a + k\delta : 1 \leq k \leq 2^{N} \} $$ for some $a \in \mathbb{R}$ and $\delta > 0$.
\end{Proposition}

\begin{proof}
If $U$ is in arithmetic progression we have easily $G_U = 2$. On the other hand if we have $G_U = 2$ since $x + \frac{1}{x} = 2$ if and only $x = 1$ we have for any $k$ (taking $\ell = 1$) $$u_{k + 2} - u_{k +1} = u_{k+1} - u_k$$ which concludes the proof.
\end{proof}

We will be particularly interested by \textit{homogeneous sets} that is to say sets $H$ of the form $$H := H_{a,N} = \{ ka :1 \leq  k \leq 2^N \}$$ for some integers $a \in \mathbb{N}^*$ and $N\in\N$. In particular, we wish to perturb a little an homogeneous set $H_{a,N}$ into a set $H'$ such that the Perron's constant of $H'$ is still bounded bounded. More precisely, fix any $a,N\in \mathbb{N}^*$ and let $\epsilon$ be an arbitrary function $$ \epsilon : H_{a,N} \rightarrow (0,1).$$ Define then the set $H_{a,N}(\epsilon)$ as $$H_{a,N}(\epsilon) := \left\{ (1+\epsilon(x)) x : x \in H_{a,N} \right\}.$$

\begin{figure}[h!]
\centering
\includegraphics[scale=1]{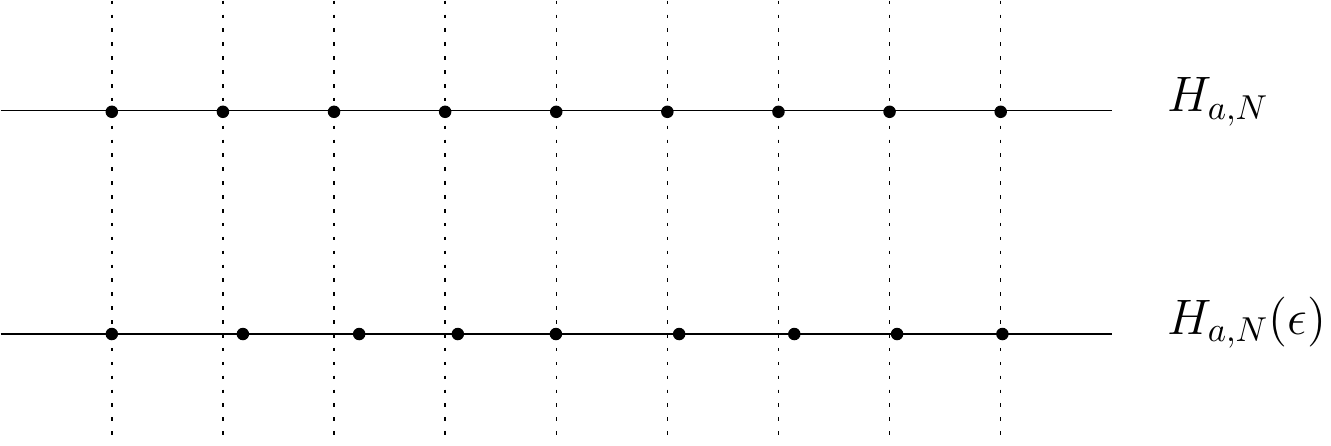}
\caption{A representation of a uniformly distributed set and its perturbation.}  
\end{figure}

\begin{Proposition}\label{P2}
Assume that $a\in\R^*_+$, $N\in\N^*$ and $\epsilon:H_{a,N}\to (0,1)$ are given.
If one has $$2^N\|\epsilon\|_\infty \leq\frac 12 ,$$ then we have $G_{H_{a,N}} \leq 6$.
\end{Proposition}
\begin{proof}
To simplify the notations, let for $1\leq k\leq 2^N$, $h_k:=[1+\epsilon(ka)]ka$.
First, observe that for $1\leq k\leq 2^N-1$, one has:
\begin{multline*}
h_{k+1}-h_k=\{1+\epsilon[(k+1)a]\}(k+1)a-[1+\epsilon(ka)]ka\\
\geq a-ak\epsilon(ka)=a[1-k\epsilon(ka)]\geq a[1-(2^N-1)\|\epsilon\|_\infty]>0,
\end{multline*}
so that there holds $h_1<h_2<\cdots <h_{2^N}$.

Now compute, for $1\leq l\leq k$ such that $k+2l\leq 2^N$ (implying in particular that one has $\frac kl \leq 2^N-2$):
$$
\frac{h_{k+2l}-h_{k+l}}{h_{k+l}-h_k}\leq \frac{la+2^Na\|\epsilon\|_\infty}{la-\|\epsilon\|_\infty ka}=\frac{1+\frac{2^N}{l}\|\epsilon\|_\infty}{1-\frac kl \|\epsilon\|_\infty}\leq \frac{1+2^N\|\epsilon\|_\infty}{1-(2^N-2)\|\epsilon\|_\infty}\leq 3.
$$ One obtains a similar inequality for the symmetric ratio.
\end{proof}

\section{Proof of Theorem \ref{T1}}

Recall that we have defined $$\Omega_{ \boldsymbol{e}} := \left\{ \frac{n}{\cos n}:n \in \mathbb{N}^* \right\}$$ and that we wish to prove that the Perron capacity of $\Omega_{ \boldsymbol{e}}$ is finite \textit{i.e.} that we have $P(\Omega_{ \boldsymbol{e}}) < \infty$. To do so, we are simply going to prove that the set $\Omega_{ \boldsymbol{e}}$ contains ``small perturbations'' of \emph{arbitrarily long homogeneous sets.} Specifically, for any $N \in \mathbb{N}$, we consider the set $$E(N) := \{ n \in \mathbb{N}^* : \exists m \in \mathbb{Z}, |n + 2\pi m| < 2^{-N } \}.$$ To begin with, we claim the following.

\begin{Claim}
For any $N\in\N$ we have $\#E(N) = \infty$.
\end{Claim}

\begin{proof}[Proof of the claim]
The claim follows from the fact that $G' = \left\{ n + 2\pi m : n\in \mathbb{N}, m \in \mathbb{Z} \right\}$ is dense in $\mathbb{R}$, which can be shown easily using standard techniques from basic real analysis\footnote{See \emph{e.g.} https://lefevre.perso.math.cnrs.fr/PagesPerso/enseignement/Archives/SSgrpesAdd.pdf {P.~Lef\`evre's webpage} for a proof.}.\end{proof}
\begin{Claim}\label{claim2}
For any $n\in\N$ and any $n\in E(N)$, one has:
$$
1<\frac{1}{\cos n}\leq 1+2^{-2N}.
$$
\end{Claim}

\begin{proof}[Proof of the claim]
Choosing $m\in\Z$ such that one has $|n+2\pi m|<2^{-N}$ (which exists by definition since one has $n\in E(N)$), one obtains, using the inequality $1-\cos x\leq \frac 12 x^2$ valid for any $x\in\R$:
$$
1-\cos n=1-\cos (n+2\pi m)\leq\frac 12 (n+2\pi m)^2\leq \frac 12 2^{-2N}.
$$
Note that, in particular, this yields $0<\cos n<1$.

Using the fact that one has $\frac{1}{1-x}\leq 1+2x$ for any $0\leq x\leq \frac 12$, one finally computes:
$$
\frac{1}{\cos n}=\frac{1}{1-(1-\cos n)}\leq 1+2(1-\cos n)\leq 1+2^{-2N},
$$
as was announced.
\end{proof}

Fix an arbitrary large $N \in \mathbb{N}$ and any integer $a \in E\left(2N\right)$. We claim the following.

\begin{Claim}
For any $a \in E(2N)$, we have $$H_{a,N} \subset E(N).$$
\end{Claim}

\begin{proof}
Fix $a \in E(2N)$ ; by definition we have $m \in \mathbb{Z}$ such that $$\left| a +2\pi m\right| < 2^{-2N}.$$ Hence for any $1 \leq k \leq 2^N$ we have $$\left| ka +2\pi km \right| < k2^{-2N} \leq 2^{-N}$$ \textit{i.e.} 
for any $1 \leq k \leq 2^N$ we have $ka \in E(N)$. 
\end{proof}

Fix now $N\in\N^*$ and $a\in E(2N)$ and define $\epsilon:H_{a,N}\to (0,1)$ by the formula $1+\epsilon(n)=\frac{1}{\cos n}$ for $n\in H_{a,N}$~---~we hence see the set:
$$
\left\{\frac{n}{\cos n}:n\in H_{a,N}\right\}=\left\{[1+\epsilon(n)] n:n\in H_{a,N}\right\},
$$
as a perturbation of the above type of the set $H_{a,N}$~---~. We compute, using Claim~\ref{claim2}:
$$
2^N\|\epsilon\|_\infty\leq 2^{-N}\leq\frac 12 ,
$$
so that Proposition~\ref{P2} yields $G_{H_{a,N}(\epsilon)}\leq 6$.

Finally, observe that we have by construction the inclusion $$H_{a,N}(\epsilon) \subset \Omega.$$ Since this holds for any $N \in\N^*$, it follows that we have $P_\Omega \leq 6$ as desired.

\qed

 \begin{Remark} The same conclusion holds if in Theorem \ref{T1} we replace $\Omega_{ \boldsymbol{e}} := \left\{ \frac{\cos n}{n} : n \in \mathbb{N}^* \right\}$ with 
 $$\Omega_{ \boldsymbol{s}} := \left\{ \frac{\sin n}{n} : n \in \mathbb{N}^* \right\}.$$ 
 As an application of Bateman's Theorem, we can hence say that also the set $\Omega_{ \boldsymbol{s}}$ is not finitely lacunary. The argument is the same but, in this case, we approximate $\frac{\pi}{2}$ by elements of  $G'$ and take $1 + \epsilon(x) =  \frac{1}{\sin x}$.

\end{Remark}

\subsection*{Acknowledgements.} This research has been supported by the \emph{\'Ecole Normale Sup\'erieure}, Paris, by the \emph{``Gruppo Nazionale per l'Analisi Matematica, la Probabilità e le loro Applicazioni dell'Is\-tituto Nazionale di Alta Matematica F. Severi''} and by the project \emph{Vain-Hopes} within the program \emph{Valere} of \emph{Università degli Studi della Campania ``Luigi Vanvitelli''}. It has also been partially accomplished within the \emph{UMI} Group \emph{TAA ``Approximation Theory and Applications''}.\\
E.~D'Aniello would like to thank the \emph{\'Ecole Normale Sup\'erieure}, Paris, for its warm hospitality in March-April, 2022; L.~Moonens would like to thank the \emph{Dipartimento di Matematica e Fisica} of the \emph{Universit\'a degli Studi della Campania ``Luigi Vanvitelli''} for its warm hospitality in April-May, 2022.

\bibliographystyle{plain}
\bibliography{DGM1}

\vspace{0.3cm}
\noindent
\small{\textsc{Emma D'Aniello}}\\[0.1cm]
\small{\textsc{Dipartimento di Matematica e Fisica},}
\small{\textsc{Universit\`a degli Studi della Campania ``Luigi Vanvitelli''},}
\small{\textsc{Viale Lincoln n. 5, 81100 Caserta,}}
\small{\textsc{Italia}}\\
\footnotesize{\texttt{emma.daniello@unicampania.it}.}

\vspace{0.3cm}
\noindent
\small{\textsc{Anthony Gauvan}}\\[0.1cm]
\small{\textsc{Laboratoire de Math\'ematiques d'Orsay, Universit\'e Paris-Sud, CNRS UMR8628, Universit\'e Paris-Saclay},}
\small{\textsc{B\^atiment 307},}
\small{\textsc{F-91405 Orsay Cedex},}
\small{\textsc{Fran\-ce}}\\
\footnotesize{\texttt{anthony.gauvan@universite-paris-saclay.fr}.}

\vspace{0.3cm}
\noindent
\small{\textsc{Laurent Moonens}}\\[0.1cm]
\small{\textsc{Laboratoire de Math\'ematiques d'Orsay, Universit\'e Paris-Saclay, CNRS UMR 8628},}
\small{\textsc{B\^atiment 307},}
\small{\textsc{F-91405 Orsay Cedex},}
\small{\textsc{Fran\-ce}}\\
\footnotesize{\texttt{laurent.moonens@universite-paris-saclay.fr}.}

\noindent\small{\textsc{\'Ecole Normale Sup\'erieure-PSL University, CNRS UMR 8553},}
\small{\textsc{45, rue d'Ulm},}
\small{\textsc{F-75230 Parix Cedex 3},}
\small{\textsc{Fran\-ce}}\\
\footnotesize{\texttt{laurent.moonens@ens.fr}.}

\end{document}